\documentclass[a4paper, 11pt]{article}

\usepackage{mathrsfs,amssymb,amsmath, amsthm,color,tikz-cd,tikz,epic,leftindex,nicematrix,blkarray}

\usepackage[top=1.31in,bottom=1.29in,left=1.1in,right=0.9in]{geometry}
\usepackage[numbers]{natbib}
\usepackage[all]{xy}

\usepackage[utf8]{inputenc}
\usepackage[T1]{fontenc}
\usepackage{adigraph}

\setcitestyle{open={},close={}}

\numberwithin{equation}{section}\theoremstyle{definition}
\swapnumbers

 \newtheorem{Theorem}[equation]{Theorem}
 \newtheorem{Prop}[equation]{Proposition}
 \newtheorem{Lemma}[equation]{Lemma}
 
 \newtheorem{Notation}[equation]{Notation}
 \newtheorem{Defn}[equation]{Definition}
 \newtheorem{Example}[equation]{Example}
 \newtheorem{Remark}[equation]{Remark}
 \newtheorem{Remarks}[equation]{Remarks}

\newtheorem{Cons}[equation]{Consequences}

\newtheorem{Facts}[equation]{Facts}

 \makeatletter
\def\enumerate{\begingroup\ifnum\@enumdepth>3\@toodeep\else
      \advance\@enumdepth\@ne
      \edef\@enumctr{enum\romannumeral\the\@enumdepth}%
      \topsep\z@\parskip\z@
      \list{\csname label\@enumctr\endcsname}
        {\@nmbrlisttrue\let\@listctr\@enumctr
         \parsep\z@\itemsep\z@\topsep\z@
         \setcounter{\@enumctr}{0}
         \def \fMakelabel##1{\hss\llap{\rm ##1}}
       }\fi}

 \makeatother

\def\t{\mathfrak t}

\def\fs{\mathfrak s}

\def\fR{\mathfrak R}
\def\fS{\mathfrak S}
\def\fT{\mathfrak T}

\def\C{\mathbb C}
\def\N{\mathbb N}

\def\Z{\mathbb Z}
\def\F{\mathbb F}

\def\cO{\mathcal O}
\def\cH{\mathcal H}

\def\cR{\mathcal R}
\def\cP{\mathcal P}

\def\cE{\mathcal E}

\def\cI{\mathcal I}

\def\cC{\mathcal C}

\def\cS{\mathcal S}

\def\cD{\mathcal D}

\def\cA{\mathcal A}

\def\mP{\mathscr P}

\def\mE{\mathscr E}

\def\L{\text{\tiny$\mathbf L$}}

\DeclareMathOperator{\Mod}{Mod}
\DeclareMathOperator{\Hom}{Hom}
\DeclareMathOperator{\GL}{GL}

\DeclareMathOperator{\Res}{Res}
\DeclareMathOperator{\Ind}{Ind}

\DeclareMathOperator{\End}{End}

\DeclareMathOperator{\Stab}{Stab}

\DeclareMathOperator{\im}{im}

\DeclareMathOperator{\Infl}{Infl}
\DeclareMathOperator{\Inv}{Inv}

\def\comp{\models}

\def\boj{\mathbf{j}}
\def\boc{\mathbf{c}}

\def\rV{V^{\otimes r}}
\def\setn{\{1,\ldots,n\}}
\def\setr{\{1,\ldots,r\}}

\def\la{\lambda}
\def\lak{(\lambda_1,\ldots,\lambda_k)}
\def\muk{\{\mu_1,\ldots,\mu_k\}}
\def\pik{(\pi_1,\ldots,\pi_k)}
\def\ck{(c_1,\ldots,c_k)}

\def\jr{(j_1,\ldots,j_r)}

\def\HRn{\cH_{R,q}(\fS_n)}
\def\part{\cP_{R,q}(n,r)}

\def\IR{[IR]}
\def\cIR{[\cI\cR]}
\def\hH{\hat{H}}
\def\rhalf{{r+\frac{1}{2}}}
\def\parth{\cP_{R,q}(n,r+{\textstyle\frac{1}{2}})}

\DeclareMathOperator{\lt}{\ell}

\DeclareMathOperator{\first}{first}

\begin{document}

\title{Iwahori-Hecke algebras acting on tensor space by $q$-deformed letter permutations and $q$-partition algebras}

\author{Geetha Thangavelu $^{*}$, Richard Dipper$^{**}$\\ \\$^{*}${\footnotesize 
Indian Institute of Science Education and Research} \\ {\footnotesize Thiruvananthapuram  695551 India}
\\ \scriptsize{E-mail:tgeetha@iisertvm.ac.in}\smallskip \\$^{**}$ {\footnotesize Institut f\"{u}r Algebra und Zahlentheorie}\\ {\footnotesize Universit\"{a}t Stuttgart, 70569 Stuttgart, Germany}
\\ \scriptsize{E-mail: richard.dipper@mathematik.uni-stuttgart.de}\\
\setcounter{footnote}{-1}\footnote{
{\scriptsize 
The first author's research was partially supported by DST-SERB under grants MATRIX MTR/2017/000424 and Power SPG/2021/004200. The first author would like to extend her gratitude to the Humboldt Foundation and Prof. Steffen Koenig for their support of her stay at University of Stuttgart.}
{\scriptsize }}
\setcounter{footnote}{-1}
\setcounter{footnote}{-1}}
\date{}


 \maketitle

\begin{abstract}  Let $R$ be a  commutative ring with identity and let $V$ be a free $R$-module of rank $n$ for some $n\in\N$. Fixing an $R$-basis $\cE$ of $V$, the symmetric group $\fS_n$ acts on $V$ by permuting $\cE$ and hence on tensor space $\rV$ for $r\in\N$ via the usual tensor product action turning $V$ and $\rV$  into $R\fS_n$-modules. For units $q$ in $R$ we construct an action of the corresponding Iwahori-Hecke algebra 
$\cH_{R,q}(\fS_n)$ which specializes to the action of $R\fS_n$,  if $q$ is taken to $1$. The centralizing algebra of this action is called the $q$-partition algebra 
$\part$. Let $R$ be a field of characteristic not dividing $q$. We prove, that $\part$ is isomorphic to the $q$-partition algebra defined by Halverson and Thiem by different means a few years ago.  

\end{abstract}

\section{Introduction}  Throughout rings will have an identity and for rings $A$ and $B$, we denote the category of left $A$-modules by $\leftindex_{A}{\Mod}$, of right $B$-modules by $\Mod_B$ and $A$-$B$-bimodules by $\leftindex_{A}{\Mod}_B$.

Let $R$ be a commutative ring, $V$ a free $R$-module of rank $n\in\N$, and for $r\in\N$ let $\rV = V\otimes_R V\otimes_R\cdots\otimes_R V$. We fix an $R$-basis $\cE= \{e_1,\ldots,e_n\}$ of $V$, then 
$$
\cE^r = \{e_{i_1}\otimes\cdots\otimes e_{i_r}\in \rV\,|\,i_1,\ldots,i_r\in\setn\}
$$
is an $R$-basis of $\rV$. 
The symmetric group $\fS_r$ on $r$ letters acts on $\rV$ by {\bf place permutations} permuting the tensor factors making $\rV$ an $R\fS_n$-module. Then obviously $R\fS_r$ centralizes the action of the group algebra $RGL_n(R)$, where the general linear group $GL_n(R)$ acts on $V$ and hence diagonally on $\rV$ in the usual way. In fact for infinite fields $R$ it is well known that the actions of the group algebras $RGL_n(R)$ and $R\fS_r$  on $\rV$ satisfy Schur-Weyl duality. That is their images in $\End_R(\rV)$ are mutual centralizing algebras of each other. Indeed, this holds under some mild restrictions for general commutative rings $R$, see [\cite{Cruz}]. Moreover the action of $R\fS_r$ on $\rV$ is the special case $q=1$ of a $q$-deformed action of the Iwahori-Hecke algebra $\cH_{R,q}(\fS_r)$ on $\rV$ due to Jimbo, see [\cite{jimbo}]. 

 On the other hand defining $w(e_{i_1}\otimes\cdots\otimes e_{i_r})= e_{i_{w1}}\otimes\cdots\otimes e_{i_{wr}}$ for $1\leq i_1,\ldots,i_r\leq n$ and $w\in\fS_n$, the symmetric group $\fS_n$ on $n$ letters acts by {\bf letter permutations} on $\cE^r$. Extending this by linearity defines an $R\fS_n$-module structure on $\rV$, which coincides with the one arising by restricting the natural diagonal action of $GL_n(R)$ on $\rV$ to $\fS_n\leq GL_n(R)$. However this action of $R\fS_n$ on $\rV$ does not carry over directly  to $\cH_n = \cH_{R,q}(\fS_n)$, since $\cH_n$ does not have a Hopf coproduct allowing it to act on the tensor product $\rV$. Moreover the embedding of $\fS_n$ into $GL_n(R)$ has no $q$-analogue replacing $GL_n(R)$ by some quantized version of general linear groups.  Still there exists a more complicated  
$\cH_n$-action on $\rV$ which specializes to the letter action of $R\fS_n$ on $\rV$ by taking $q=1$, as we will show in section 3 of this paper. This result is based on the observation of Halverson and Ram,  that $\rV \in\leftindex_{R\fS_n}{\Mod}$ can be obtained by iterated  restriction to $\fS_{n-1}\leq\fS_n$ and induction to $\fS_n$ of the trivial $R\fS_n$- module (see [\cite{halv ram},3.19]). This construction can easily be extended to the Iwahori-Hecke algebra $\cH = \cH_{R,q}(\fS_n)$ and exploited to define an action of $\cH$ on $\rV$ and decompose tensor space into a direct sum of $q$-permutation modules, see theorem \ref{Hecke on tensor space} below. The main theme of our investigation here concerns the endomorphism algebra of this $\cH$-action. In the classical situation of the $\fS_n$-action, this is isomorphic to the partition algebra introduced by  Martin [\cite{martin 1},\cite{martin 2},\cite{martin 3}] and Jones [\cite{jones 1}] provided $n\geq 2r$. In general, the endomorphism algebra 
$ \End_{R\fS_n}(\rV)$ is an epimorphic image of some partition algebra. 

Our version of the $q$-partition algebra  $\part$ is defined  to be the endomorphism ring of the $\cH_{R,q}(\fS_n)$-module $\rV$. Note that this construction works for arbitrary commutative rings $R$ and arbitrary units $q\in R$. Indeed all constructions are stable under base change. Thus one can work over ring $R = \Z[t,t^{-1}]$ of integral Laurent polynomials in the variable $t$ and then prove the results for arbitrary rings $R$ and units $q$ by specialisation. In particular one obtains 

\begin{equation}  
\part \cong R\otimes_{\Z[t,t^{-1}]}\cP_{\Z[t,t^{-1}],q}(n,r).
\end{equation}

Halverson and Thiem defined in [\cite{halv thiem}] a $q$-analogue $Q_r (n, q)$ of the partition algebra $P_r(n)$ to be the endomorphism algebra of  a certain $\C GL_n(q)$-module $\fT_Q^r$. This is obtained by  iterating  $r$ times  Harish–Chandra restriction and induction   between $GL_{n-1}(q)$ and $ GL_n(q)$ applied to the trivial $GL_n(q)$-module $\leftindex_{GL_n(q)}{\C}$. Here $GL_n(q)$ is the finite general linear group over the finite field $\F_q$ with $q$ elements, where $q$ is some prime power. Thus $Q_r(n,q)$ is defined only for prime powers $q$. 

 The main result of this paper is then: 

\begin{Theorem}Suppose that  $q\in\C$ is a prime power and let $n,r\in\N$. Then  
$$
\cP_{\C,q}(n,r)\cong Q_r(n,q).
$$
\hfill$\square$\end{Theorem}

The isomorphism in the theorem is obtained by utilizing so called Hecke functors, studied in [\cite{dipper1}] and [\cite{dipper2}], which connect the representation theory of finite general linear groups with the representations of Iwahori-Hecke algebras. These constructions do not work for arbitrary commutative coefficient rings $R$. To avoid subtle ring theoretic considerations we restrict the rings $R$ in this part of the paper to be one of the rings in an $l$-modular system for primes $l$ not dividing $q$.   

Halverson and Thiem also generalised the notion of half-integer partition algebra. Our approach works for this as well, that is we obtain half integer $q$-partition algebras $\parth$  for arbitrary commutative rings $R$ and units $q\in R$. 

Finally we prove that the action of the Iwahori-Hecke algebra on tensor space by $q$-deformed letter permutations satisfies Schur-Weyl duality applying a result of Donkin [\cite{donkin 2}], provided $R$ is a Dedekind domain or a field. Moreover for fields $R$ cellularity of $q$-partition algebras follows most likely by extending  Donkin's work [\cite{donkin 1}] to the $q$-deformed situation and this work allows then as well to determine precisely, when the $q$-partition algebra is quasi-hereditary.

\section{Preliminaries}

Throughout we fix $n,r\in \N$. The symmetric group on some set $M$ is denoted by $\fS_M$. For $M = \{1, \ldots,n\}$ we write also $\fS_n$  instead of $ \fS_M$. A \textbf{composition} $\la = (\la_1, \ldots, \la_k)$ of $n$, written as $\la\models n$, is a finite sequence of nonnegative integers $\la_i$ whose sum equals $n$.  If the sequence $\la$ is weakly decreasing we call $\la$ a \textbf{partition} of $n$ and write $\la \vdash n$. For $\la=\lak\models n$ the \textbf{standard Young subgroup} $Y_\la$ of $\fS_n$ consists of those permutations of $\setn$ leaving invariant the subsets $\{1,\ldots,\la_1\}, \{\la_1+1,\la_1+2,\ldots,\la_1+\la_2\}, \{\la_1+\la_2+1,\ldots\},\ldots$.Thus $Y_\la$ is isomorphic to the direct product $\fS_{\la_1}\times\fS_{\la_2}\times\cdots\times\fS_{\la_k}$.

For the notion of a $\la$-tableau $\t$ the reader is referred e.g. to [\cite{james}]. For $\la\comp n$ let $\t^\la$ be the {\bf initial $\la$-tableau} in which the numbers $1,2,\ldots,n$ appear in order from left to right along successive rows. For example, if $\la = (3,0,2,4)\comp 9$ then 

\begin{equation}
\t^\la\,\, = \,\,\begin{matrix}1&2&3&&\\-&-&-&-&\\4&5&&&\\6&7&8&9\end{matrix}
\end{equation}

The set of $\la$-tableaux, $\la\comp n$,  is denoted by $\fT (\la)$.  The symmetric group $\fS_n$ acts transitively on $\fT (\la)$ by letter permutation. Hence for any $\fs\in\fT(\la)$ there exists precisely one $d(\fs)\in\fS_n$ with

\begin{equation}\label{d of s}
d(\fs)\t^\la = \fs.
\end{equation}

A $\la$-tableau $\fs\in\fT(\la)$ is called \textbf{row standard}, if the entries in the rows of $\fs$ increase from left to right. The set of row standard $\la$-tableaux is denoted by $\fR(\la)$.

Let $R$ be a commutative ring and let $V$ be the free $R$-module of rank $n$. Fix an $R$-basis $\cE= \{e_1,\ldots,e_n\}$ of $V$. Then $\cE^r = \{e_{i_1}\otimes\cdots\otimes e_{i_r}\in \rV\,|\,i_1,\ldots,i_r\in\setn\}$ is an $R$- basis of the tensor space $\rV$. We write $I(n,r)$ for the set of multi-indices 
$\boj = (j_1, \ldots, j_r)$ with $j_\nu\in\setn, \nu = 1,\ldots,n$. Moreover we set $e_\boj = e_{j_1}\otimes\cdots\otimes e_{j_r}\in\cE^r$ and obtain $\cE^r = \{e_\boj\,|\,\boj \in I(n,r)\}$.

The symmetric group $\fS_n$ acts on $\setn$ from the left by letter permutations and hence on $\cE$ by $w(e_i) = e_{w(i)}$ for $i\in\setn$ and $w\in\fS_n$. Moreover it acts  on $I(n,r)$ by $w(\boj) = w(j_1,\ldots,j_r) = (w(j_1), \ldots, w(j_r))$ and on $\cE^r$ by $w(e_\boj) = e_{w(\boj)}$, for $\boj = (j_1,\ldots,j_r)\in I(n,r)$. Extending  this by linearity we obtain an action of the symmetric group algebra  $R\fS_n$ on tensor space $V^{\otimes r}$. Indeed this is the restriction to $\fS_n\leq \GL_n(R)$ of the full linear group $\GL_n(R)$  acting diagonally on the tensor space. Recall that the symmetric group $\fS_r$ on $r$ letters acts from the right on $V^{\otimes r}$ by place permutations. 

We now collect some basic well known facts on the $R\fS_n$-module structure of $\rV$, see for example [\cite{ben halv}]. Being the permutation module of $R\fS_n$ with underlying $\fS_n$-set $\cE^r$, the $R\fS_n$-module $\rV$ decomposes into a direct sum of orbit modules $R\cO$ corresponding to the decomposition of $\cE^r$ into $\fS_n$-orbits $\cO$. The orbits are described by {\bf colored set partitions} of $\setr$ as follows:

 \begin{Defn}\label{defsetpart} We consider set partitions $\mP = \pik, k\in\N$ of $\setr$ with $ k\leq \min(n,r)$  in pairwise disjoint non-empty subsets $\pi_i$ of $\setr$ called \textbf{blocks},  whose union is $\setr$. We identify the elements of the blocks with {\bf places} (i.e. relative positions) of tensor factors  in pure
 tensors of $\rV$. Moreover the $\pi_i$  are ordered  such that the smallest places $b_i \in \pi_i$ for $ i=1,\ldots, k$ satisfy $b_1<b_2<\cdots<b_k$. Thus $1\in\pi_1$, the smallest number $1\leq s\leq r$ not contained in $\pi_1$ is contained in $\pi_2$, the smallest number not contained in $\pi_1\cup\pi_2$ is contained in $\pi_3$  and so on. A \textbf{coloring} $\mP^\boc$ of $\mP$ consists of a map  assigning to each block $\pi_i$ a number $c_i\in\setn$, called \textbf{color} of $\pi_i$ for $i=1,\ldots,k$ such that $c_i=c_j$ only if $i=j$.
\hfill$\square$\end{Defn}

The restriction $k\leq r$ is obvious and $k\leq n$ is needed to ensure that different blocks can be colored differently. For $\boc = \ck$ and $\mP = \pik$ as above we denote the corresponding colored set partition by  $\mP^\boc$ and the collection of colorings of $\mP$ by  $\cC = \cC(\mP)$.

\begin{Lemma}\label{action colors} Let $\mP=\pik$ as above. Then $\fS_n$ acts on $\cC = \cC(\mP)$ by $w\boc = w\ck = (w(c_1),\ldots,w(c_k))$ for $\boc = \ck\in\cC$ and $w\in\fS_n$. Moreover $\{w\boc\,|\,w\in\fS_n\} = \cC$ for any $\boc\in\cC$. Thus $\fS_n$ acts transitively on $\cC$.
\end{Lemma}
\begin{proof} Note that the symmetric group acts $n$-fold and hence $k$-fold transitively on $\setn$ and the lemma follows.
\end{proof}

\begin{Defn} \label{def part col basis} Let $\boj = \jr\in I(n,r)$ and set $x = e_\boj = e_{j_1}\otimes\cdots\otimes e_{j_r}\in \cE^r$. Let $k$ be the number of distinct indices occuring in $x$. Then obviously $k\leq\min(n,r)$. 

For $1\leq i\leq n$ we let $\first_x(i)$ be the smallest position $s$  at which $e_i$ occurs in $x$ and set  $\first_x(i) = 0$, if there is no such position. That is
\begin{equation}\label{first}
\first_x(i) = \first_\boj(i) = 
\begin{cases}
0  & \text{if $i\not=j_s$ for $1\leq s\leq r$}\\
s & \text{if  $i=j_s$ and $i\not=j_t$ for $1\leq t< s$}.
\end{cases}
\end{equation}
 Define the set partition $\mP_x = \pik$ as follows: The nonnegative integers $1\leq s,t \leq n$ belong to the same block $\pi_i, (1\leq i\leq k)$ of $\mP_x$ if and only if $j_s = j_t$. Define the coloring $\boc(x) = \ck$ of $\mP_x$ setting $c_i = j_s$ for $s\in\pi_i$. Thus $j_s$ is the color of $\pi_i$. Obviously $e_\boj\mapsto \mP^{\boc(e_\boj)}$ sets up an bijection between  $\cE^r$ and the set of colored set partitions of $\setr$ into at most $\min(n,r)$ parts. 
\hfill$\square$\end{Defn}

\begin{Example} \label{first ex} We take $r=8$ and $n=7$. We illustrate a set partition by a graph with vertex set $\setr$, whose connected components are the blocks $\pi_i$. So for $\boj = (3,6,3,1,1,3,1,3)$ we have $e_\boj  =  e_3\otimes  e_6\otimes  e_3\otimes  e_1\otimes   e_1\otimes  e_3\otimes  e_1\otimes  e_3$. Thus $k=3$ and $\mP^{\boc(e_\boj)} = (\pi_1, \pi_2, \pi_3)$ with $\pi_1 =\{1,3,6,8\}$ of color $3$, $\pi_2 = \{2\}$ of color $6$ and $\pi_3 = \{4,5,7\})$ of color $1$.  The corresponding graph is given as

\begin{center}
\begin{tikzpicture} 

\coordinate[label=below:$1$](1)at(0,0); 
\coordinate[label=below:$2$](2)at(1,0); 
\coordinate[label=below:$3$](3)at(2,0); 
\coordinate[label=below:$4$](4)at(3,0); 
\coordinate[label=below:$5$](5)at(4,0); 
\coordinate[label=below:$6$](6)at(5,0); 
\coordinate[label=below:$7$](7)at(6,0); 
\coordinate[label=below:$8$](8)at(7,0); 

\draw[-,thick](1)to[out=45,in=135](3);
\draw[-,thick](3)to[out=45,in=135](6);
\draw[-,thick](6)to[out=45,in=135](8);
\draw[-,thick](1)to[out=45,in=135](3);
\draw[-,thick](4)to[out=45,in=135](5);
\draw[-,thick](5)to[out=45,in=135](7);

\fill(1)circle(2pt); 
\fill(2)circle(2pt);
\fill(3)circle(2pt); 
\fill(4)circle(2pt); 
\fill(5)circle(2pt); 
\fill(6)circle(2pt); 
\fill(7)circle(2pt); 
\fill(8)circle(2pt);  

\end{tikzpicture}

\end{center}
Moreover $\first_\boj(2) = \first_\boj(4) = \first_\boj(5) = \first_\boj(7) = 0$ and $\first_\boj(1)=4, \first_\boj(3)=1$ and $\first_\boj(6) = 2$.
\hfill${\square}$\end{Example}

The action of $w\in\fS_n$ on $x=e_\boj\in\cE^r$ preserves the set partition $\mP_x$ attached to $x$ and changes only (transitively) the coloring  $\boc = \boc(x)$ of it. Thus the set partion $\mP_x$ is an invariant of the $\fS_n$-orbit $\cO_x$ of $x$ under the action of $\fS_n$ and we my write $\mP_\cO$ for $\cO=\cO_x$. Since $\fS_n$ acts transitively on $\cC(\mP_x)$ we have shown:

\begin{Prop}\label{orbits} Let $x\in\cE^r$ and let $\mP = \mP(x)$. Then  $\mP = \mP_{wx}$ for all $w\in\fS_n$ and the $\fS_n$-orbit $\cO_x = \fS_nx$ of $x$ is in bijection with the collection of all colorings of $\mP$ given by $wx\mapsto \boc(wx) = w\boc(x)$. Thus the set of all set partitions $\mP$ of $\setr$ into at most $\min(n,r)$ many parts labels the $\fS_n$-orbits of $\cE^r$. Hence the number  of these orbits is given by as the sum of the Stirling numbers $s(r,k)$ for $1\leq k\leq \min(n,r)$ of the second kind.  
\hfill${\square}$\end{Prop} 

Next we inspect the permutation module attached to $\fS_n$-orbits on $\cE^r$. Thus let $\cO\subseteq\cE^r$ be such an orbit and choose $x\in \cO$. Let $\mP= \mP_\cO = \pik$ be the corresponding set partition in $k$ parts say. Thus in particular $k\leq n$. Define the partition $\la = \la(\cO) = (\la_1,\ldots,\la_{k+1})$ setting $\la_1 = n-k$ and $\la_i = 1$ for $i=2,\ldots,k+1$. We write $\la = (n-k , 1^k)$ for short. Note that $\la$ is a so called \textbf{hook partition} of $n$. 

\begin{Defn}\label{bold s of x} Let $x\in\cE^r$ with colored partition $\mP^{\boc(x)} = \pik$ and $\boc(x) = \ck$ as in \ref{def part col basis} above. 
We define the  $\la$-tableau $\fs = \fs(x)$ as follows: In the first row of $\fs$ we put the numbers $1\leq s \leq n$ which are not contained in $\{c_1,\ldots,c_k\}$ in increasing order. Then we insert the numbers $c_1,c_2,\ldots,c_k$ into rows $n-k+1, \ldots,n$ in that order.
\hfill${\square}$\end{Defn}

 By construction $\fs$ is row standard. Conversely \ref{action colors} implies that for any row standard $\la$-tableau $\t$   there exist $x\in\cO$ such that $\t = \fs(x)$. We have shown:

\begin{Prop}\label{orbitbij} The map $\cO\to\fR(\la):x\mapsto\fs(x)$ is a bijection.
\hfill${\square}$\end{Prop}

\begin{Example}\label{sec ex} For the $x\in\cE^8$ in \ref{first ex} we have $\la = (n-3,1^3) = (4,1^3) \vdash 7$ and 
$$
\fs(x) \,\, = \,\,\begin{matrix}2&4&5&7\\3&&&\\6&&&\\1&&&\end{matrix}
$$
\hfill${\square}$\end{Example}

\begin{Notation} For rings $S\subseteq T$ we denote the induction and restriction functors by 
$$
\Ind_S^T:\leftindex_S {\Mod}\to\leftindex_T{\Mod} \text{\,\,and\,\,}\Res^T_S:\leftindex_T{\Mod}\to\leftindex _S{\Mod}\text{\,\, respectively}.
$$
\hfill${\square}$\end{Notation}

Let $\cO$ again be a $\fS_n$-orbit of $\cE^r$ and let $x\in\cO$. Let the associated hook partition be $\la = (n-k,1^k)$, where $k$ is the number of different indices occuring in the pure tensor $x\in\cE^r$. Let $H = \Stab_{\fS_n}(x) = \{w\in\fS_n\,|\,wx=x\}$ be the {\bf stabilizer} of $x$ in $\fS_n$. By general theory the permutation module $R\cO$ is isomorphic to the trivial $RH$-module $\leftindex_H{R}$ induced up to $G$. By \ref{orbitbij} this is isomorphic to the $R\fS_n$-permutation module arising from the action of $\fS_n$ on the set $\fR(\la)$ of row standard $\la$-tableaus  by letter permutations followed by reordering the entries in the rows (here indeed only row $1$) in increasing order. Moreover choosing $x\in\cO$ such that $\fs(x) = \t^\la$ is the initial $\la$-tableau we obtain  that  $\Stab_{\fS_n}(x) = \Stab_{\fS_n}(\t^\la)$ is the standard Young subgroup $Y_\la \cong \fS_m = \fS_{\{1,\ldots,m\}}$ of $\fS_n$. For $\la\comp n$ and 
$Y_\la\leq\fS_n$ we denote the permutation module $\Ind_{Y_\la}^{\fS_n}(\leftindex_{Y_\la}{R})$ on the left cosets of $Y_\la$ in $\fS_n$ by $M^\la$.    

We summarize the results so far in the following theorem, (compare [\cite{ben halv}, section 3]):

\begin{Theorem}\label{Tensorspacedec} The tensor space $\rV$ decomposes as $\fS_n$-module into the direct sum 
$$
\rV = \bigoplus_{k=1}^{\min(n,r)} s(r,k) M^{(n-k,1^k)},
$$
where $s(r,k)$ is the Stirling number of the second kind.
\hfill${\square}$\end{Theorem}

In the next section we shall use \ref{Tensorspacedec} to construct an analogous, $q$-deformed action of the Iwahori-Hecke algebra $\cH_n = \HRn$ on $\rV$.

\section{The action of $\cH_n$ on $\rV$  by $q$-deformed letter permutations} 

Our main player in this section is the Iwahori-Hecke algebra $\cH = \cH_n = \HRn$ defined as $q$-deformation of the group algebra $R\fS_n$. We first collect some basic facts on $\cH$. Note that whenever we take $q$ to $1$ we obtain corresponding results for the group algebra $R\fS_n$. For details the reader is referred to [\cite{dj1}] or the lecture notes [\cite{mathas 1}] of A.  Mathas.

Denote by $S$ the set of simple transpositions $s_i = [i,i+1]\in\fS_n,\,i=1,\ldots,n-1$. Then $\fS_n$ is generated by $S$. Thus $w\in\fS_n$ can be written as product with factors in $S$. We call such an expression for $w$ {\bf reduced}, if the number of factors from $S$ is minimal and then this number is called {\bf length} of $w$ and denoted by $\lt(w)$. Indeed all minimal expressions for $w\in\fS_n$ have the same number of factors from $S$ and hence the length function $\lt:\fS_n\to\Z_{\geq 0}$ is well defined. 

Now let $q\in R$ be a unit.

\begin{Defn}\label{Iwahori-Heckealg} The {\bf Iwahori-Hecke algebra} $\cH = \HRn$ is generated as $R$-algebra with identity by $T_1,\ldots,T_{n-1}$ subject to the following relations:
$$
\begin{array}{lccll}
\text{i)} & T_iT_j &=& T_jT_i& \text{for}\,1\leq i,j\leq n-1, \, |i-j|\geq 2\\
\text{ii)} & T_iT_ {i+1}T_i&=& T_{i+1}T_iT_{i+1}& \text{for}\, 1\leq i\leq n-2\\
\text{iii)} & T_iT_i & = & q+ (q-1)T_i & \text{for}\, 1\leq i\leq n-1.
\end{array}
$$
\hfill${\square}$\end{Defn} 

For $1\leq i\leq n-1$ and $s_i = [i,i+1]\in S$ we write also $T_{s_i}=T_i$ .  If $w = t_1t_2\cdots t_k$ with $t_j\in S$ is a reduced expression for $w\in\fS_n$ we define 
$T_w = T_{t_1}T_{t_2}\cdots T_{t_k}$ .  Then $T_w$ does not depend on the choice of a reduced expression for $w$ and $\HRn$ is free as $R$-module with basis $\{T_w\,|\,w\in\fS_n\}$. Furthermore if $v,w\in\fS_n$ with $\lt(vw)=\lt(v)+\lt(w)$,  then $T_vT_w = T_{vw}$. 

\begin{Lemma}\label{inverses}
The generators  $T_s, s\in S$,  are invertible with inverse
$$ 
T_s^{-1} = (q^{-1}-1) + q^{-1}T_s.
$$
Furthermore let $w = t_1t_2\cdots t_k\in\fS_n$ be a reduced expression. Then $T_w$ is invertible with inverse
$$
 T_w^{-1} = T_{t_k}^{-1}T_{t_{k-1}}^{-1}\cdots T_{t_1}^{-1}.
$$
\hfill${\square}$ \end{Lemma}

Recall that for $w\in\fS_n$ and $s\in S$ we have $\lt(sw) = \lt(w)\pm1$ and similarly  $\lt(ws) = \lt(w)\pm1$, depending of $w$ having a reduced expression beginning  (ending respectively)  with $s$ for the negative or not for the positive sign. From this one obtains immediately:

\begin{equation}
T_sT_w = 
\begin{cases} 
T_{sw} &\quad\text{if}\, \lt(sw)=\lt(w)+1\\
qT_{sw} + (q-1)T_w & \quad\text{if}\, \lt(sw)=\lt(w)-1.
\end{cases}
\end{equation}\hfill${\square}$

Now let $\la = \lak$ be a composition of $n$. Then the subalgebra $\cH_\la = \cH_{R,q}(Y_\la)$ of $\cH$ generated by $\{T_s\,|\,s\in S\cap Y_\la\}$  is free as $R$-module with basis $\{T_w\,|\,w\in Y_\la\}$. Extending the definition of Iwahori-Hecke algebras to direct products of symmetric groups, $\cH_\la$ is the Iwahori-Hecke algebra $\cH_{R,q}(Y_\la)$ associated with the standard Young subgroup $Y_\la$ of $\fS_n$. In the special case of hook partitions  $\la = (k,1^{n-k})\vdash n$ with $1\leq k\leq n$ we write $Y_\la = \fS_{\{1,\ldots,k\}} = \fS_k\leq \fS_n$.

\begin{Lemma}\label{trivalt} Let $\la\comp n$. We set
$$
x_\la = \sum_{w\in Y_\la} T_w\in\cH_\la.
$$
Then $T_wx_\la = q^{\lt(w)}x_\la = x_\la T_w$ for all $w\in Y_\la$. Thus  $\cH_\la x_\la=x_\la\cH_\la = Rx_\la$ is a left and a right $\cH_\la$-module,  free as $R$-module of $R$-rank $1$, and  is a  $q$-deformation of  the trivial $RY_\la$-module. We hence define the $\HRn$-module
$$
 M^\la_q = \HRn x_\la
$$
and call $M^\la_q$ {\bf $q$-permutation module} to $\la\comp n$. Observe that $M^\la_q = \Ind_{\cH_\la}^\cH (Rx_\la)$ is a $q$-analogue of the permutation module $M^\la$ of $R\fS_n$ on the cosets of $Y_\la$ in $\fS_n$.
\hfill${\square}$\end{Lemma}

\begin{Remark}\label{y sub la}
There is also a $q$-analogue for the alternating module for symmetric groups: Setting 
$$
y_\la = \sum_{w\in Y_\la} (-q)^{-\lt(w)}T_w\in\cH_\la ,
$$
we have $T_wy_\la = y_\la T_w = (-1)^{\lt(w)}y_\la$ for all $w\in Y_\la$. The $1$-dimensional space $Ry_\la$ is the desired $q$-analogue for the alternating $RY_\la$-module. Note that $\cH x_\tau, \,x_\tau\cH$ and $\cH \,y_\tau, y_\tau\cH$ with $\tau = (n)\comp n$ are the only  $\cH$-modules, which are $R$-free  as $R$-modules of rank $1$ . 
\hfill${\square}$\end{Remark}

Next we collect some basic facts on $M^\la_q$. Again for details we refer to [\cite{dj1}] and [\cite{mathas 1}].

\begin{Facts}\label{q-permmods} Let $\la\comp n$ and let $\fs\in \fR(\la)$ be a row standard $\la$-tableau. By \ref{d of s} there is a unique $d(\fs)\in\fS_n$ such that $d(\fs)\t^\la = \fs$, where again $\t^\la$ is the initial $\la$-tableau. We have:
\begin{enumerate}
\item[1.)] The set $\cD_\la = \{d(\fs)\,|\,\fs\in\fR(\la)\}$ is a set of left and  $\cD_\la^{-1} = \{d^{-1}\,|\,d\in\cD_\la\}$ is a set of  right coset representatives of $Y_\la$ in $\fS_n$. These coset representatives are called {\bf distinguished}.
\item[2.)] Let $d\in\cD_\la$ and $w\in Y_\la$. Then $\lt(dw) = \lt(d)+\lt(w)$ and hence $T_dT_w=T_{dw}$. Thus $d$ is the unique element of minimal length in its left coset $dY_\la$. A similar statements holds for $d^{-1}$ and the right coset $Y_\la d^{-1}$. Moreover $d\in\cD_\la$ and $s\in S$ such that $\lt(sd)=\lt(d)-1$ implies $sd\in\cD_\la$ (see e.g. [\cite{dj1}, 4.1]).
\item[3.)] The Iwahori-Hecke algebra $\cH = \HRn$ is free as right (left) $\cH_\la$-module with basis $\{T_d\,|\,d\in\cD_\la\}$ (respectively 
$\{T_d\,|\,d\in \cD_\la^{-1}\}$).
\item[4.)] The $q$-permutation module $M_q^\la = \Ind_{\cH_\la}^\cH (Rx_\la)$ has basis $\{T_d x_\la\,|\,d\in\cD_\la\}$. Let $s = s_i = [i,i+1]\in S$ and $d\in\cD_{\la}$. Set $\fs = d\t^\la$.  Then $\fs\in\fR(\la)$ and  (see e.g.  [\cite{dj1},3.2])
$$
T_sT_d x_\la = \begin{cases}
qT_dx_\la & \quad \text{if $ i$ and $i+1$  belong to the same row of $\fs$}\\
T_{sd}x_\la & \quad \text{if the row index of $i$ in $\fs$ is less than that of $i+1$}\\
qT_{sd}x_\la + (q-1)T_dx_\la & \quad \text{otherwise}.
\end{cases}
$$
\item[5.)]  Now choose a further $\mu \comp n$. Then $\cD_{\mu,\la} = \cD_\mu^{-1}\cap\cD_\la$ is a set of $Y_\mu$-$Y_\la$ double coset representatives. Again $d\in\cD_{\mu,\la}$ is the unique shortest element in the double coset $Y_\mu d Y_\la$ and hence is called {\bf distinguished} as well. Moreover $ d^{-1}Y_\mu d\cap Y_\la$ is a standard Young subgroup $Y_\tau$ of $Y_\la$ and of $\fS_n$ for some $\tau\comp n$.  Henceforth we write $\tau = \mu d\cap\la$. As a consequence the subalgebra 
$$
\cH^{T_d}_\mu\cap \cH_\la = T_d^{-1}\cH_\mu T_d\cap \cH_\la
$$ 
of $\HRn$ is the subalgebra $\cH_\tau$ associated with the composition $\tau = \mu d\cap\la$ of $n$ (see e.g. [\cite{dj1},1.6]).
\item[6.)] Let $M\in\leftindex_{\cH_\mu}{\Mod}$. Then we have a Mackey decomposition theorem with respect to standard Young subalgebras, (see e.g. [\cite{dj1},2.7]):
\begin{equation}\label{Mackey}
\Res_{\cH_\la}^\cH\Ind_{\cH_\mu}^\cH (M)\cong \bigoplus_{d\in\cD_{\mu,\la}} \Ind_{\cH_{\mu d\cap \la}}^{\cH_\la}(\Res^{\cH_\mu}_{\cH_{\la d^{-1}\cap \mu}}M)^{T_d},
\end{equation}
where $-^{T_d}$ means that we twist the action of  $\cH_{\la d^{-1}\cap \mu}$ on $M$ by conjugation by $T_d$ to obtain an operation on $M$ by 
$ \cH_{\mu d\cap\la}$.
\item[7.)] Furthermore by [\cite{dj1},3.4] the set $\{\varphi_d\,|\, d\in\cD_{\mu,\la}\}$ is an $R$-basis of $\Hom_\cH(M_q^\mu, M_q^\la)$, where 
\begin{equation}\label{qSchurbasis}
\varphi_d:M_q^\mu\to M_q^\la: x_\mu\mapsto \sum_{d'\in\cD_{\mu\cap\la d^{-1}}\cap Y_\mu} T_{d'd}x_\la = \sum_{w\in Y_\mu d Y_\la}T_w.
\end{equation}
 \end{enumerate}
Observe that the combinatorial tools in all these constructions above are independent of the choice of the unit $q$ in $R$, and hence in particular the data derived from those as for instance dimensions  coincide for the group algebras of the symmetric groups and the associated Iwahori-Hecke algebras. In particular the $\Hom$-spaces between  $q$-permutation modules of the form $M_q^\la, \,\la\comp n$ allow base change, that is change of the ring $R$ of coefficients.
\hfill${\square}$\end{Facts}

In \ref{Tensorspacedec} we have seen that tensor space $\rV$ is isomorphic to  a direct sum of permutation modules of the form $M^\la$ for some hook partitions $\la$, which in turn are isomorphic as $R$-modules to $M_q^\la$ by part $4$.) of \ref{q-permmods}. Thus

\begin{equation}\label{R isom tensorspace}
\rV \cong \bigoplus_{k=1}^{\min(n,r)} s(r,k) M^{(n-k,1^k)}_q
\end{equation}
as $R$-module. We may use this $R$-isomorphism to transfer the $\HRn$-action on the right hand side module in \ref{R isom tensorspace} to $\rV$. To carry this out it suffices to determine the action of the generators $T_i, i=1,\ldots,n-1$ of $\cH=\HRn$ on the basis $e_\boj\in\cE^r$ for $\boj\in I(n,r)$.  Thus let $\boj = \jr \in I(n,r)$ and $x = e_\boj =  e_{j_1}\otimes\cdots\otimes e_{j_r}\in \cE^r$. Recall the definitions of the colored set partition $\mP^{\boc(x)} = \pik$ in \ref{def part col basis} with coloring $\boc(x) = \ck$ and of $\fs = \fs(x)$ in \ref{bold s of x}. Moreover recall the definition of $\first_x(i) = \first_\boj(i)$ in \ref{first}. Note that $\fs$ is a Young tableau to some hook partition and hence the only row containing possibly more than one entry is the first one.  Now part 4.) of  \ref{q-permmods}  translates immediately into:

\begin{Lemma}\label{tensor diagram} Let $1\leq i\leq n-1$. Then 
\begin{enumerate}
\item[i)] $i$ and $i+1$ belong to the same row of $\fs$ if and only if $\first_x(i) = \first_x(i+1)=0$. 
\item[ii)] The row index of $i$ in $\fs$ is less than that of $i+1$ if and only if $\first_x(i) < \first_x(i+1)$.
\item[iii)] The row index of $i$ in $\fs$ is greater than that of $i+1$ if and only if $\first_x(i) > \first_x(i+1)$.

\end{enumerate}
\hfill${\square}$\end{Lemma}

Note that in case i) above necessarily $\first_\boj(i) = \first_\boj(i+1) = 0$. Now the following theorem follows immediately from \ref{q-permmods}  and \ref{tensor diagram}:

\begin{Theorem}\label{Hecke on tensor space} The Iwahori-Hecke algebra $\cH = \HRn$ acts on tensor space $\rV$, where the action of the generator $T_i$ of $\cH$ for $1\leq i\leq n-1$ on 
$x = e_\boj \in \cE^r$ for $\boj\in I(n,r)$ is given as
$$
T_i x = 
\begin{cases}
qx & \text{if} \first_x(i) = \first_x(i+1) =0.  \\                                        
y & \text{if}  \first_x(i) < \first_x(i+1) \\
qy + (q-1)x & \text{if}  \first_x(i) >\first_x(i+1),
\end{cases}
$$
where $y=  e_{[i,i+1]\boj}\in \cE^r$. Moreover 
$$
\rV \cong \bigoplus_{k=1}^{\min(n,r)} s(r,k) M_q^{(n-k,1^k)}\,\,\text{as $\cH$-module,}
$$
where $s(r,k)$ is the Stirling number of the second kind.
\hfill${\square}$\end{Theorem}

\begin{Remarks}\label{relationcheckremark} \begin{enumerate}
\item[1.)] Note that the action of $\HRn$ on $\rV$ in \ref{Hecke on tensor space} is indeed a $q$-deformation of the permutation operation of the symmetric group $\fS_n$ on tensors in $\cE^r$ by letter permuations. In fact, the latter is recovered by taking $q$ to $1$ in  \ref{Hecke on tensor space}. As mentioned in the introduction, there is a different $q$-deformation of the action of the symmetric group $\fS_r$ on $\rV$ by place  permutations due to Jimbo [\cite{jimbo}].
\item[2.)] One can prove theorem \ref{Hecke on tensor space} directly, by checking, that the action of $T_i$ on $\cE^r$ as described there satisfies the relations  i), ii) and iii) of the definition \ref{Iwahori-Heckealg}.
\item[3.)] Note that the action of the symmetric group on tensor space utilizes the fact, that the group algebra $R\fS_n$ has a Hopf structure with coproduct yielding a diagonal embedding of the group  into the direct product of $r$ copies of $\fS_n$. There is no Hopf algebra structure on $\cH$ yielding  the action of $\cH$ on tensor space here.\hfill${\square}$
\end{enumerate}
\end{Remarks}



\section{The $q$-partition algebra}

 In the early 1990s partition algebras $P_r(Q)$   were introduced by Martin [\cite{martin 1},\cite{martin 2},\cite{martin 3}] and Jones [\cite{jones 1}] in their work on Temperley-Lieb algebra and the Potts model in statistical mechanic. These algebras are defined over arbitrary commutative rings $R$  in diagrammatical form, having a basis of set-partition diagrams and multiplication given by diagram concatenation involving the parameter $Q$. For $Q=n = \dim_R(V),$ the partition algebra $P_r(n)$ acts on tensor space $\rV$ centralizing the diagonal action of $\fS_n$ on $\rV$ by letter permutations. 
We have a surjective algebra homomorphism from $P_r(n)$ to $\End_{R\fS_n}(\rV)$, which in addition turns out to be an isomorphism for $n\geq 2r$, (see e.g. [\cite{halv jacobson},2.5]).

Obviously theorem \ref{Hecke on tensor space} yields  a $q$-deformation of the endomorphism ring $\End_{R\fS_n}\rV$, by defining:

\begin{Defn}\label{q-partition algebra} Let again $R$ a commutative ring, $ n,r\in\N$ and let $q\in R$ be a unit. Then define the {\bf$q$-partition algebra} $\part$ by
$$
 \part = \End_\cH(\rV) \,\,\text{where again\,} \cH =  \HRn.
$$
\hfill${\square}$\end{Defn}  

Of course, this definition is stretching the notion of partition algebras in the case of $n<2r$, but is correct otherwise. In  [\cite{halv thiem}],  Halverson and Thiem defined their $q$-partition algebras $Q_r(n,q)$ for $R=\C$ and prime powers $q$  by different means. These come with the same restriction, their construction provides only a true $q$-analogue of classical partition algebras $P_r(n)$ if $n\geq 2r$.  Recall that our definition works for any commutative ring $R$ and unit $q\in R$. The goal in this section is to extend the Halverson Thiem construction to any field $R$ of characteristic not dividing $q$ and to some special integral domains $R$ and to  prove, that for prime powers $q$ both versions over $R$ of  the $q$-partition algebra are isomorphic. In order to establish this isomorphism  we need to investigate the construction of Halverson and Thiem more in detail.

The starting point is an observation of Halverson and Ram in [\cite{halv ram}] that tensor space $\rV$ can also  be constructed by iterated restriction to and induction from $\fS_{n-1}$  of the trivial $R\fS_n$-module. Indeed this is  the special case $q=1$ of the corresponding $q$-analogue  for the Iwahori-Hecke algebra $\cH = \cH_n = \HRn$ as we shall show next. 

To simplify notation we write
 \begin{equation}\label{inductionrestriction}
\IR_{n-1}^n = \Ind_{\cH_{\la}}^{\cH}\Res_{\cH_{\la}}^{\cH},
\end{equation}
where $\la = (n-1,1)\vdash n$. Note that then $\fS_{n-1}\cong Y_\la\leq \fS_n$.
With $\tau= (n)\vdash n$ we have $Y_\tau = \fS_n$ and hence $Rx_\tau$ is the trivial $\cH$-module.   

\begin{Theorem}\label{indres}  Keep the notation introduced above.   Then: 
$$
\rV \cong (\IR_{n-1}^n)^r (Rx_{(n)}),
$$
as $\HRn$-modules.
\end{Theorem}
\begin{proof} The proofs of this result in [\cite{halv ram},3.19] and [\cite{halv thiem},1.2] in the special case $q=1$ are based  on well known tensor identities, which do not carry over to Iwahori-Hecke algebras. We outline here a proof using Mackey decomposition, which works for $\cH$ as well by part 6.) of \ref{q-permmods}. We proceed by induction on $r$ to show 
\begin{equation}\label{assertion}
(\IR_{n-1}^n)^r (Rx_{(n)}) \cong \bigoplus_{k=1}^{\min(n,r)} s(r,k) M^{(n-k,1^k)}_q
\end{equation}
and then \ref{Hecke on tensor space} implies the assertion immediately. 

Set again $\la = (n-1,1)\vdash n$. For $r=1$ we have  $\IR_{n-1}^n (Rx_{(n)}) =\Ind_{\cH_{\la}}^{\cH}\Res_{\cH_{\la}}^{\cH}(Rx_{(n)})  = M_q^{\la}$ by \ref{trivalt} as desired.


Now let $2\leq r\in\N$ and suppose we have shown the assertion for $r-1$. Thus by induction hypothesis
\begin{equation}\label{indhypo}
\begin{split}
(\IR_{n-1}^n)^r (Rx_{(n)}) &\cong \IR_{n-1}^n(V^{\otimes(r-1)})\\
&\cong \IR_{n-1}^n\left( \bigoplus_{k=1}^{\min(n,r-1)} s(r-1,k) M^{(n-k,1^k)}_q\right)\\
&\cong  \bigoplus_{k=1}^{\min(n,r-1)} s(r-1,k)\IR_{n-1}^n( M^{(n-k,1^k)}_q).\\
\end{split}
\end{equation}

Thus we have to investigate $\IR_{n-1}^n( M^{(n-k,1^k)}_q)$. But
\begin{equation}\label{hyposummand}
\IR_{n-1}^n( M^{(n-k,1^k)}_q) = \Ind_{\cH_\la}^{\cH}\Res_{\cH_\la}^{\cH}\Ind_{\cH_\mu}^{\cH}(Rx_\mu),
\end{equation}
setting $\mu = (n-k,1^k)\vdash n$. We apply Mackey decomposition (see part 6 of \ref{q-permmods}) to obtain:
\begin{equation}\label{mackeyaplohne}
\begin{split}
\Res_{\cH_\la}^{\cH}\Ind_{\cH_\mu}^{\cH} (Rx_\mu)
&\cong\bigoplus_{d\in\cD_{\mu,\la}} \Ind_{\cH_{\mu d\cap \la}}^{\cH_\la}(\Res^{\cH_\mu}_{\cH_{\la d^{-1}\cap \mu}}Rx_\mu)^{T_d}\\
&\cong \bigoplus_{d\in\cD_{\mu,\la}}  \Ind_{\cH_{\mu d\cap \la}}^{\cH_\la}(\Res^{\cH_\mu}_{\cH_{\la d^{-1}\cap \mu}}Rx_\mu)^{T_d}\\
&\cong  \bigoplus_{d\in\cD_{\mu,\la}}\Ind_{\cH_{\mu d\cap \la}}^{\cH_\la}(Rx_{\mu d\cap\la})\\
\end{split}
\end{equation}
since  restriction to standard Young subalgebras and conjugation by $T_d$ takes trivial modules to trivial ones.


 Recall that the set $S$ of basic transpositions $s_i = [i,i+1]$ generates $\fS_n$.  It is not hard to see that 
  $\cD_\la = \{d_1,\ldots,d_{n-1},d_n=1\}$, where  $d_i = s_is_{i+1}\cdots s_{n-1}$ for $i=1,\ldots,n-1$, and that $d_i\in\cD_\mu^{-1}$ if and only if 
$n-k\leq i\leq n$, since $Y_\mu d_j = Y_\mu d_{n-k}$ for $1\leq j\leq n-k-1$. We have shown 
\begin{equation}\label{double cosets}
\cD_{\mu,\la}=  \cD_\mu^{-1}\cap\cD_\la =
\begin{cases}
\{d_{n-k},d_{n-k+1},\ldots,d_{n-1},d_n=1\} & \quad \text{for\,} k<n\\
\{d_{n-k+1},\ldots,d_{n-1},d_n=1\} = \cD_\la & \quad\text{for\,}k=n. 
\end{cases}
\end{equation}
Note that in particular $|\cD_{\mu,\la}|=k+1$ for $k<n$ and  $|\cD_{\mu,\la}|=n$ for $k=n$.  Moreover, for $n-k+1\leq i\leq n-1$ we have $s_iw=ws_i$ for all $w\in Y_\mu$ and hence $Y_\mu^{d_i}\cap Y_\la = Y_\mu\leq Y_\la$. Thus $\mu d\cap\la = \mu$ in this case. Moreover one checks immediately, that $d_{n-k}^{-1}Y_\mu d_{n-k}\cap Y_\la = Y_\tau$ with $\tau = (n-k-1, 1^{k+1})$, that is $\mu d_{n-k}\cap\la = \tau$.

 We write henceforth $M^\mu_q (\la)=  \Ind_{\cH_\mu}^{\cH_\la}(Rx_\mu)$ for any $\mu\comp n$ with $Y_\mu\leq Y_\la$.
Then \ref{mackeyaplohne} becomes
\begin{equation}\label{Mackey 2}
\begin{split}
\Res_{\cH_\la}^{\cH}\Ind_{\cH_\mu}^{\cH} (Rx_\mu)
&\cong  \bigoplus_{d\in\cD_{\mu,\la}}\Ind_{\cH_{\mu d\cap \la}}^{\cH_\la}(Rx_{\mu d\cap\la})\\
& = \bigoplus_{k\,\text{copies}}M_q^{(n-k,1^k)}(\la)\oplus M_q^{(n-k-1,1^{k+1})}(\la)\\
& = kM_q^{(n-k,1^k)}(\la)\oplus M_q^{(n-k-1,1^{k+1})}(\la).\\
\end{split}
\end{equation}

We insert this into equation \ref{indhypo} and apply the identities 
\begin{align*}
s(r,k) &= k\cdot s(r-1,k)+s(r,k-1)\\
s(r,0) &= 0\\
s(r,1) &= 1 
\end{align*}
 for the Stirling numbers of the second kind to obtain:
\begin{equation}\label{Mackey 4}
\begin{split}
\Res_{\cH_\la}^\cH(\IR_{n-1}^n)^{r-1} (Rx_{(n)})
&\cong   \bigoplus_{k=1}^{\min(n,r-1)} s(r-1,k)  \Res_{\cH_\la}^{\cH}\Ind_{\cH_\mu}^{\cH} (Rx_\mu)\\
&\cong \bigoplus_{k=1}^{\min(n,r-1)} s(r-1,k)( kM_q^{(n-k,1^k)}(\la)\oplus M_q^{(n-k-1,1^{k+1})}(\la))\\
&= \bigoplus_{k=1}^{\min(n,r)}(k\cdot  s(r-1,k)+ s(r-1,k-1))M_q^{(n-k,1^k)}(\la)\\
&=\bigoplus_{k=1}^{\min(n,r)}s(r,k)M_q^{(n-k,1^k)}(\la),
\end{split}
\end{equation}
and hence, observing $\Ind_{\cH_\la}^\cH\left(M_q^{(n-k,1^k)}(\la)\right) = M_q^{(n-k,1^k)},$ we obtain

\begin{equation}\label{Mackey 5}
(\IR_{n-1}^n)^r (Rx_{(n)}) =\Ind_{\cH_\la}^\cH\Res_{\cH_\la}^\cH(\IR_{n-1}^n)^{r-1} (Rx_{(n)})=
 \bigoplus_{k=1}^{\min(n,r)}s(r,k)M_q^{(n-k,1^k)}.
\end{equation}
 
Now the theorem follows by induction and \ref{Hecke on tensor space}.
\end{proof}

Next we outline the construction of the $q$-partition algebra $Q_r(n,q)$ defined by Halverson and Thiem. For this, a third player enters the stage, the finite general linear group $G=GL_n(q)$ defined over the field $\F_q$ with $q$ elements, where now $q$ is some  power of some prime $p$. Although in  [\cite{halv thiem}] Halverson and Thiem define their $q$-partition algebra over the field of complex numbers their construction carries over to fields of characteristic $l\geq 0$ not dividing $q$ and even to commuative coefficient rings $R$ satisfying some restrictions  as long as $q\cdot 1_R$ is a unit of $R$. However to avoid ring theoretic subtleties we choose a prime $p$ not dividing $q$ and a split $l$-modular system  $(F,\cO,K)$ for $G$ such that $\cO$ is a rank one complete discrete valuation ring with field of fractions $K$ of characteristic $0$  and residue field $F$ of characteristic $l$ such that $K$ and $F$ are splitting fields for all subgroups of $G$. Moreover let $R$ henceforth be one of the rings $F,\cO,K$.   

We list some basic facts on $G=GL_n(q)$, which we will need later on. The group $G$ is a finite group of Lie type and hence is a group with an $BN$-pair, a fact which we will use. For details on those groups we refer to the standard literature  see e.g. [\cite{carter}] or [\cite{curtis reiner}]. 

 We choose the subgroup of invertible upper triangular matrices in $G$ to be $B$, (a so called {\bf Borel subgroup}). The subgroup $U^+$ of $B$ of upper unitriangular matrices is normal in $B$ and is a $p$-Sylow subgroup of $G$. Subgroups of $G$ which contain $B$ are called {\bf standard parabolic subgroups}. Thus $B\leq P\leq G$ be a parabolic subgroup of $G$. Then $P$ admits a so called {\bf Levi decomposition} $P = L\rtimes U_P$, where $U_P$ is the unipotent radical of $P$ and $L$ is the {\bf Levi complement} and $P$ is the semi direct product of $U_P$ by $L$. Each parabolic subgroup is conjugate to a parabolic subgroup containing the Borel subgroup $B$.  The parabolic subgroups $P\geq B$ are in bijection with compositions of $n$: If $P \leftrightarrow \mu = \muk\comp n$ we write $P = P_\mu$ with Levi decomposition $P_\mu =L_\mu U_\mu$. Here $L_\mu \cong GL_{\mu_1}\times GL_{\mu_2}\times\cdots\times GL_{\mu_k}$ embedded into $GL_n(q)$ in block form down the diagonal. $U_\mu$ consists then of upper unitriangular matrices, whose entries above the diagonal are zero, if they belong to the blocks of $L_\mu$.

\begin{Defn}\label{HC-functors}       
Let $\mu\comp n$. Define  the functor {\bf Harish-Chandra induction} 
$$
\cI_{L_\mu}^G  = \Ind_{P_\mu}^G\circ\Infl_{L_\mu}^{P_\mu}:\leftindex_{RL_\mu}{\Mod} \to \leftindex_{RG}{\Mod}
$$
 where the {\bf inflation functor} $\Infl_{\L_\mu}^{P_\mu}$ turns $RL_\mu$-modules into $RP_\mu$-modules by the natural surjection $P_\mu \to L_\mu$ with kernel $U_\mu$. Moreover define the functor {\bf Harish-Chandra restriction}
$$
\cR_{L_\mu}^G = \Inv_{U_\mu}\circ\Res^G_{P_\mu}:\leftindex_{RG}{\Mod} \to\leftindex_{RL_\mu}{\Mod}
$$
where $\Inv_{U_\mu} (N)$ for $N\in\leftindex_{RP_\mu}{\Mod}$ is defined to be the set of $U_\mu$-invariants in $N$, which obviously is an $RL_\mu$-submodule of $N$. Since $l\not=p$,  the order $|U_\mu|$ of $U_\mu$ is invertible in $\cO$ and $\Inv_{U_\mu}(N)$  can be realised  as $ e_{U_\mu}N$ with $e_{U_\mu} = |U_\mu|^{-1} \sum_{u\in U_\mu}u\in\cO G$.
\hfill${\square}$\end{Defn}

Harish-Chandra induction and restriction play a very prominent role for the representation theory of finite general linear groups and more generally of groups of Lie type.  They are an adjoint pair on both sides and 
 Harish-Chandra inductions (respectively restrictions) are isomorphic functors for compositions $\nu, \mu$ of $n$, if $\mu$  arises from $\nu$ by reordering its parts. As a consequence it suffices to consider parabolic subgroups $P_\mu = L_\mu U_\mu$ with $\mu\vdash n$.  
We collect some basic facts, for the general setting and details we refer e.g. to [\cite{difl1},\cite{difl2}].
 Extending the definition of these functors to direct products of general linear groups we have functors 
\begin{equation}
\cI_{L_\nu}^{L_\mu} : \leftindex_{RL_\nu}{\Mod}\to\leftindex_{RL_\mu}{\Mod} \,\,\text{and}\,\,
\cR_{L_\nu}^{L_\mu}: \leftindex_{RL_\mu}{\Mod}\to\leftindex_{RL_\nu}{\Mod}
\end{equation}
for $\nu,\mu\comp n$ such that $L_\nu\leq L_\mu$. Then these functors satisfy transitivity, that is  given $\tau, \nu,\mu\comp n$ with $L_\tau\leq L_\nu\leq L_\mu$  we have 

\begin{equation}\label{HC trans}
\cI_{L_\nu}^{L_\mu}\circ\cI_{L_\tau}^{L_\nu} \cong \cI_{L_\tau}^{L_\mu}\quad\text{and}\,\, \cR^{L_\nu}_{L_\tau}\circ \cR^{L_\mu}_{L_\nu} \cong \cR^{L_\mu}_{L_\tau}.
\end{equation}

We have a  Mackey decomposition theorem similar to  part 6.) of \ref{q-permmods}, (comp [\cite{didu2}, 1.4]:

\begin{equation}\label{HC Mackey}
\cR_{L_\nu}^G\cI_{L_\mu}^G (N)\cong \bigoplus_{d\in\cD_{\mu,\nu}} \cI_{L_{\mu d\cap \nu}}^{L_\nu}(\cR^{L_\mu}_{L_{\nu d^{-1}\cap \mu}}N)^d\,\,\text{for}\, N\in\leftindex_{RL_\mu}{\Mod}.
\end{equation}
 Here $\fS_n$ is considered as the subgroup of permutation matrices in $G$.

Henceforth fix $\la = (n-1,1)\vdash n$. Then in particular
$$
L_\la = \left\{
\begin{pNiceArray}{cw{c}{1cm}c|c}[margin]
\Block{3-3}<\large>{GL_{n-1}(q)}&&&0 \\
&&&\Vdots\\
&&&0\\
\hline
0 & \Cdots&0&*
\end{pNiceArray}
\in GL_n(q)\right\}
\,\,\text{and}\,\,
U_\la = \left\{
\begin{pNiceArray}{cw{c}{1cm}c|c}[margin]
\Block{3-3}<\large>{E_{n-1}}&&&* \\
&&&\Vdots\\
&&&*\\
\hline
0 & \Cdots&0&1
\end{pNiceArray}
\in GL_n(q)\right\},
$$
where $E_{n-1}$ is the identity matrix of $GL_{n-1}(q)$. 
We now have the tools to define the $q$-partition algebra [\cite{halv thiem}] of Halverson and Thiem:

\begin{Defn}\label{q-part halv thiem} The subgroup $ P_\la$ of $G$ is a maximal parabolic subgroup  containing $B$ and $L_\la\cong GL_{n-1}(q)$. Write $\cIR_{n-1}^n = \cI_{L_\la}^G\circ\cR_{L_\la}^G$. Then the $RG$-module  $\fT^r_{R,q}= \fT^r_q$ is defined to be 
$$
\fT^r_q = \fT^r_{R,q} = (\cIR_{n-1}^n)^r(\leftindex_G{R}),
$$
where $\leftindex_{G}{R}$ denotes the trivial $RG$-module. Then according to Halverson and Thiem the {\bf $q$-partition algebra}  $Q_r(n,q)$ is defined to be the endomorphism algebra $\End_{RG}(\fT^r_q)$.
\hfill${\square}$
\end{Defn}

We will show, that the $q$-partition algebra $Q_r(n,q)$ with $R\in\{F,\cO,K\}$  is indeed isomorphic to $\part$ defined in \ref{q-partition algebra}. But this requires some preparation.

There is a close connection between the representation theory of finite general linear groups in non describing chracteristic and Iwahori Hecke algebras. The key observation, due to Iwahori  [\cite{iwahori}], is the following result:

\begin{Theorem}\label{hecke gln} The endomorphism algebra $\End_{RG}(\Ind_B^G(\leftindex_B{R}))$ of the trivial $RB$-module $\leftindex_B{R}$ induced to $G$ is the Iwahori Hecke algebra  $ \cH_{R,q}( \fS_n)$. Thus, writing $RG$-endomorphisms on the right, we obtain the bimodule $\Ind_B^G(\leftindex_B{R}) \in\leftindex_{RG}{\Mod}_{\cH_R}$.
\hfill${\square}$\end{Theorem}

We remark in passing, that this occurence of Iwahori Hecke algebras here is only one of many. Indeed,  attached to each so called geometric conjugacy class of irreducible characters of $G$ there is attached a  specific Iwahori Hecke algebra which is a tensor product of algebras of the form $\cH_{R,q^{i}}$ for certain powers $q^{i}$ of $q$. The one presented in the theorem above is attached to the so called principal series for
 $G$. Note that $B= P_\rho$ for $\rho = (1^n)\vdash n$ with Levi decomposition $B=L_\rho U_\rho$, where $L_\rho = T$ is the subgroup of invertible diagonal matrices in $G$ and $U_\rho = U^+$. Furthermore $\Ind_B^G(\leftindex_B{R}) = \cI_T^G(\leftindex_T{R})$. Henceforth we set $N = N_R =  \cI_T^G(\leftindex_T{R})$ for $R\in\{F,\cO,K\}$. We extend this to standard Levi subgroups $L_\mu$ for arbitrary $\mu\comp n$ observing that $T\leq L_\mu$ and indicate this by a superscript $\mu$. Thus $N_R^\mu = \cI_T^{L_\mu}(\leftindex_T{R})$ is a $RL_\mu$-$\cH_\mu$-bimodule, that is  $N_R^\mu \in\leftindex_{RL_\mu}{\Mod}_{\cH_\mu}$.   

Next we apply theorem \ref{hecke gln} to construct a pair of useful functors between $\leftindex_{RG}{\Mod}$ and $\leftindex_{\cH}{\Mod}$, which at the end will produce the desired result on $q$-partition algebras. 
We start with a minimal projective cover $$
\hat{\beta}: \widehat{Q}\to \leftindex_T{\cO}
$$ 
of the trivial $\cO T$-module $\leftindex_T{\cO}$.
Applying Harish-Chandra induction takes projective covers to projective covers. Thus
$$
\beta :Q\to N_\cO = \cI_T^G(\leftindex_T{\cO})  \,\,\text{with}\,\, \beta = \cI_T^G(\hat{\beta})\,\,\text{and}\,\,Q =  \cI_T^G(\widehat{Q})
$$ 
is a (not necessarily minimal) projective cover of $N_\cO$. Let $\mE = \End_{\cO G}(Q)$. Then by  [\cite{dipper1},2.2]  $\phi(\ker \beta)\subseteq \ker\beta$ for all $\phi\in\mE$ and hence 
$J_\beta = \{\psi\in\mE\,|\,\im\psi\leq\ker\beta\}$ is an ideal in $\mE$. Moreover $\beta$ induces an $\cO$-algebra isomorphism $\mE/J_\beta \to \cH_\cO = \End_{\cO G}(N_\cO)$. For $R\in\{F,K\}$ we have $\mE_R =  R\otimes_\cO\mE_\cO\cong \End_{RG}(Q_R)$ with $Q_R = R\otimes_\cO Q_\cO$. Moreover $\cH_R = R\otimes_\cO\cH_\cO \cong \End_{RG}(N_R)$. 

\begin{Defn}\label{functors} Keeping the notation introduced above we define {\bf Hecke functors} $H$ and $\hH$  as follows:
\begin{align*}
H & = H_R:\leftindex_{RG}{\Mod}\to\leftindex_{\cH_R}{\Mod}: V\to \Hom_{RG}(Q_R,V)/J_\beta\Hom_{RG}(Q_R,V)\\
\hH & =\hH_R:\leftindex_{\cH_R}{\Mod}\to \leftindex_{RG}{\Mod}: X\to    N_R\otimes_{\cH_R}X.\\
\end{align*}
Thus $H_R = \Hom_{RG}(Q_R,-)/J_\beta\Hom_{RG}(Q_R,-)$ and $\hH_R = N_R\otimes_\cH -$.
For $\mu\comp n$ we extend the notion of the Hecke functors to Levi subgroups $L_\mu$, and indicate this by a superscript $\mu$:
$$
H_R^\mu:\leftindex_{RL_\mu}{\Mod}\to \leftindex_{\cH_\mu}{\Mod}\,\,\text{and}\,\,\hH^\mu = N_R^{\mu}\otimes_{\cH_{\mu}}-:\leftindex_{\cH_\mu}{\Mod}\to\leftindex_{RL_\mu}{\Mod},
$$
with $N_R^\mu = \cI_T^{L_\mu}(\leftindex_T{R})$.
\hfill${\square}$\end{Defn}
 In the following most properties work fine under base change with respect to the rings $R=F,\cO$ and $K$,  and we may drop the subscript $R$ unless ambiguities arise. We list now some facts, which will be needed later on,  and give arguments and references for these:

\begin{Facts}\label{Hecke facts} Let $\mu\comp n$. The Hecke functors defined in \ref{functors} satisfy:
\begin{enumerate}
\item[1.)] The functor $H^\mu$ is a left inverse of the functor $\hH^\mu$. Thus for $X\in\leftindex_{\cH_\mu}\Mod$, the $\cH_\mu$-module 
$$
H^\mu\circ\hH^\mu(X) = \Hom_{RL_\mu}(Q^\mu,N^\mu\otimes_{\cH_\mu}X)/
J_{\beta^\mu}^\mu\Hom_{RL_\mu}(Q^\mu,N^\mu\otimes_{\cH_\mu} X)
$$ 
is naturally isomorphic to $X$, by [\cite{dipper1},2.16].
\item[2.)] For left ideals $J$ of  $\cH_\mu$ we have $\hH^\mu(J) =  N^\mu\otimes_{\cH_\mu}J = N^\mu J$. In particular  $N^\mu\otimes_{\cH_\mu} \cH_\mu x\cong N^\mu x$ for $x\in\cH_\mu$, (see [\cite{dipper1},2.20] and [\cite{dipper2},5.3]).
\item[3.)] The the $\cO$-lattice $\hH(\cH_\cO x_\mu) = Nx_\mu$ is a pure sublattice of $N$, that is there exists an
 $\cO$-lattice complement of $ Nx_\mu$ in $N$, (see [\cite{james 2},7.19]). This implies in particular that $ N_Fx_\mu \cong F\otimes_\cO x_\mu N_\cO$ is a reduction modulo $l$ of 
$N_Kx_\mu $.
\item[4.)] Now let $\tau = (n)\comp n$. Then obviously $x_\tau = \sum_{w\in\fS_n}T_w$ and  $\cH x_\tau$ is the trivial $\cH$-module. Now
 $\hH(\cH x_\tau) = \leftindex_G{R}$,  the trivial $RG$-module,  for every choice of $R\in\{F,\cO,K\}$. For $R = K$ this follows from  [\cite{curtis et.al.},4.6]. Since 
$ N_\cO x_\tau$ is pure in $N_\cO$ the claim follows for the other choices of $R= F, \cO$ as well. 
This extends to Levi subgroups, that is $\hH^\mu (\cH_\mu x_\mu) =  \hH^\mu(Rx_\mu) = N^\mu x_\mu = \leftindex_{L_\mu}{R}$ for all choices of $R=F,\cO,K$.

\item[5.)] Let $\mu\comp n$. Then 
$$\cI_{L_\mu}^G(\leftindex_{L_\mu}{R})\cong \cI_{L_\mu}^G(N^\mu x_\mu) 
\cong \cI_\mu^G(\cI_T^{L_\mu}(\leftindex_T{R}))x_\mu 
\cong \cI_T^G(\leftindex_T{R})x_\mu = Nx_\mu
$$ 
for all choices of $R = F,\cO,K$.

 \item[6.)] Let $X,Y$ be left ideals of $\cH$. Then $NX = \hH(X), NY = \hH(Y)$ and the $R$-linear map induced (and denoted) by $H$  
$$
H:\Hom_{RG}(NX,NY)\to \Hom_{\cH}(H\hH(X),H\hH(Y)) = \Hom_{\cH}(X,Y)
$$
is an isomorphism. In particular 
\begin{equation}\label{basiciso}
H:\End_{RG}(NX)\to\End_{\cH}(X)  
\end{equation}
is an $R$-algebra isomorphism, (see [\cite{dipper1},2.21]).

\item[7.)] Harish-Chandra restriction and induction commute with the Hecke functor $\hH$, that is the following diagrams commute:  
\begin{center}
\begin{tikzcd}[row sep = huge]
\leftindex_\cH{\Mod}\arrow[r, "\hH"  ]& \leftindex_{RG}{\Mod}\\
\leftindex_{\cH_{\mu }}{\Mod}\arrow[u,  "\Ind_{\cH_\mu}^\cH" ]\arrow[r,  "\hH^\mu"] & \leftindex_{RL_\mu}{\Mod}\arrow[u,   "\cI_{L_\mu}^G"]
\end{tikzcd}
\hspace {2cm}
\begin{tikzcd}[row sep = large]
\leftindex_\cH{\Mod}\arrow[r, "\hH"  ]\arrow[d,  "\Res_{\cH_\mu}^\cH" ]& \leftindex_{RG}{\Mod}\arrow[d,   "\cR_{L_\mu}^G"]\\
\leftindex_{\cH_\mu}{\Mod}\arrow[r,  "\hH^\mu"] & \leftindex_{RL_\mu}{\Mod}
\end{tikzcd}
\end{center}
\bigskip
Thus, for $X \in\leftindex_{\cH_\mu}{\Mod}$ and $A\in \leftindex_{\cH}{\Mod}$  we have natural isomorphisms
$$
N\otimes_\cH\Ind_{\cH_\mu}^\cH(X)\cong \cI_{L_\mu}^G(N^\mu\otimes_{\cH_\mu}X)\quad\text{and}\quad 
N^\mu\otimes_{\cH_\mu}\Res_{\cH\mu}^\cH(A) \cong \cR_{L_\mu}^G(N\otimes_\cH A)
$$
respectively isomorphisms of functors
\begin{equation}\label{functorisos}
\hH \Ind_{\cH_\mu}^\cH \cong \cI_{L_\mu}^G \hH^\mu \quad \text{and} \quad \hH^\mu \Res^\cH_{\cH_\mu}\cong \cR_{L_\mu}^G\hH.
\end{equation}
For a proof see [\cite{didu},1.3.2]. \hfill${\square}$
\end{enumerate}
\end{Facts}

Here are two consequences of \ref{Hecke facts}:

\begin{Cons}\label{folg}
 Let $\mu\comp n$. 

\begin{enumerate}
\item[1.)] The functor $\hH$ takes the permutation type $\cH$-module $M^\mu_q$ to permutation modules on the cosets of $P_\mu$: 
\begin{alignat*}{2}
\hH(M_\mu^q) &\cong\hH(\cH x_\mu) &&\\
&\cong \hH \Ind_{\cH_\mu}^\cH (R x_\mu)&&\quad\text{by \ref{trivalt}}\\
&\cong \cI_{L_\mu}^G\hH^\mu(R x_\mu)&&\quad\text{by \ref{functorisos}}\\
&\cong \cI_{L_\mu}^G(\leftindex_{L_\mu}{R})&&\quad\text{by part 4.) of \ref{Hecke facts}}\\
&\cong\Ind_{P_\mu}^G(\leftindex_{P_\mu}{R})&&\quad\text{by \ref{HC-functors}}
\end{alignat*}
for all choices of $R = F,\cO,K$.
\item[2.)] Applying \ref{functorisos} repeatedly gives
$$
\hH\Ind_{\cH_\mu}^\cH\Res_{\cH_\mu}^\cH \cong \cI_{L_\mu}^G\hH^\mu\Res_{\cH_\mu}^\cH \cong \cI_{L_\mu}^G\cR_{L_\mu}^G\hH,
$$
and hence in particular
\begin{equation}\label{consequ}
\hH ([IR]_{n-1}^n)^r \cong([\cI\cR]_{n-1}^n)^r\hH,
\end{equation} 
and 
\begin{equation}\label{consequplushalf}
\hH^\mu\Res^\cH_{\cH_\mu} ([IR]_{n-1}^n)^r \cong\cR^G_{L_\mu}([\cI\cR]_{n-1}^n)^r\hH,
\end{equation} 

for all choices of $R = F,\cO,K$.\hfill${\square}$
\end{enumerate}
\end{Cons}


We now are in the position to prove our main results:

\begin{Theorem}\label{main 1} Set $G = GL_n(q)$. For any choice of $R\in\{F,\cO,K\}$ we have:
\begin{enumerate}
\item[1.)]  The Hecke functor $\hH$ maps tensor space to $\fT^r_q$, that is $\hH(\rV) = \fT^r_q$.
\item[2.)] The Hecke functor $H$ maps $\fT^r_q$ to tensor space: $H(\fT^r_q) = \rV$.
\item[3.)] The $RG$-module $\fT^r_q$ decomposes  into a direct sum of of permutation modules on the cosets of parabolic subgroups as follows:
\begin{equation}\label{qtensordec}
\begin{split}
\fT^r_{R,q} = \fT^r_q 
& \cong \bigoplus_{k=1}^{\min(n,r)}s(r,k)Nx_{(n-k,1^k)} \\
& \cong  \bigoplus_{k=1}^{\min(n,r)}s(r,k) \cI_{L_{(n-k,1^k)}}^G(R_{L_{(n-k,1^k)}})\\
& \cong \bigoplus_{k=1}^{\min(n,r)}s(r,k) \Ind_{P_{(n-k,1^k)}}^G(R_{P_{(n-k,1^k)}})\\
\end{split}
\end{equation}
where again $s(r,k)$ is the Stirling number of the second kind.
\item[4.)] The Hecke functor $H$ induces an $R$-algebra isomorphism 
$$
H: Q_r(n,q) =\End_{RG}(\fT^r_q)\to\part = \End_\cH(\rV).
 $$
Thus the $q$-partition algebras defined in [\cite{halv thiem}] by Halverson and Thiem and the one defined in \ref{q-partition algebra} coincide.
\end{enumerate}
\end{Theorem}
\medskip

\begin{proof}
By \ref{indres} we have $\rV \cong ( \IR_{n-1}^n)^r(Rx_{(n)})$. Applying $\hH$ gives with \ref{consequ}
$$
\hH(\rV) = \hH\circ\IR_{n-1}^n)^r(Rx_{(n)}) \cong ([\cI\cR]_{n-1}^n)(\hH(Rx_{(n)}) \cong  ([\cI\cR]_{n-1}^n)(\leftindex_G{R}),
$$
observing part 4.) of \ref{Hecke facts}. Thus part 1.) is shown. Part 2.) follows immediately, since $H$ is a left inverse of $\hH$.  Part 3.) follows by applying $\hH$ to the decomposition 
\ref{Hecke on tensor space} of $\fT^r_q$ and observing \ref{folg}. 

Now as $R$-module
\begin{equation*}
\begin{split}
\End_{RG}(\fT^r_q) 
& = \End_{RG}\left(\bigoplus_{k=1}^{\min(n,r)}s(r,k)Nx_{(n-k,1^k)}\right)\\ 
& = \End_{RG}\left(\bigoplus_{k=1}^{\min(n,r)}s(r,k)N\cH x_{(n-k,1^k)}\right)\\
& = \bigoplus_{k,m=1}^{\min(n,r)}s(r,k)s(m,k)\Hom_{RG}(N\cH x_{(n-k,1^k)}, N\cH x_{n-m,1^m)}).\\
\end{split}
\end{equation*}
This $R$-module is mapped by part 6.) of \ref{Hecke facts} under $H$ isomorphically to 
\begin{multline*}
\bigoplus_{k,m=1}^{\min(n,r)}s(r,k)s(m,k)\Hom_{\cH}(\cH x_{(n-k,1^k)}, \cH x_{n-m,1^m)})\\ = \End_\cH\left(\bigoplus_{k,m=1}^{\min(n,r)}s(r,k)s(m,r)\Hom_{\cH}(\cH x_{(n-k,1^k)}, \cH x_{n-m,1^m)}\right)
=\End_\cH(\rV).
\end{multline*}
By functoriality of $H$ this map respects multiplication showing part 4.). This finishes the proof of the theorem.
\end{proof}

\begin{Remark} Part 3.) of \ref{main 1} can also be seen directly using Mackey decomposition \ref{HC Mackey} for the Harish-Chandra functors. Indeed observing the Mackey decomposition formula \ref{HC Mackey} for Harish-Chandra restriction and induction the proof of formula \ref{assertion} carries over to an alternate proof of the formula 3.) in the theorem \ref{main 1} almost word by word replacing $\IR_{n-1}^n$ there by $\cIR_{n-1}^n$ and $Rx_{(n)}$ by $\leftindex_G{R}$.  
\hfill${\square}$\end{Remark}

In [\cite{halv thiem}] Halverson and Thiem  generalised the {\bf half integer partition algebra} 
$$
P_\rhalf(n) = \End_{R\fS_{n-1}}(\Res^{\fS_n}_{\fS_{n-1}}(\rV)) )
$$ 
to  {\bf half integer $q$-partition algebras} $Q_{\rhalf}(n,q)$ setting (comp. definition \ref{q-part halv thiem})
\begin{equation}\label{half q}
Q_{\rhalf}(n,q) =  \End_{RL_\la}(\cR^G_{L_\la}(\fT^r_q)) = \End_{RL_\la}(\cR^G_{L_\la}\cIR_{n-1}^n(\leftindex_G{R})),
\end{equation}
where $G = GL_n(q)$ and $ \la=(n,n-1)\vdash n$.

Now equation \ref{consequplushalf} implies immediately, setting $\parth = \End_ {\cH_\la}(\Res^\cH_{\cH_\la}\rV) $:

\begin{Theorem}\label{main2} The Hecke functor $H$ induces an $R$-algebra isomorphism 
$$
H: Q_{\rhalf}(n,q) =\End_{RL_\la}(\cR^G_{L_\la}\fT^r_q)  \to\End_ {\cH_\la}(\Res^\cH_{\cH_\la}\rV) = \parth,
$$
for any choice of $R\in\{F,\cO,K\}$, setting $\la=(n-1,1)\vdash n$.
\hfill${\square}$\end{Theorem}

Thus the half integer $q$-partition  algebra $Q_{\rhalf}(n,q) \cong \parth$  can as well be realised as endomorphism ring of the restriction to $\cH_\la$ of tensor space $\rV$.

Finally we observe that the decompositions of tensor space $\rV$ and the $RG$-module $\fT^r_q$ into permutation type modules in theorems \ref{indres} respectively  \ref{main 1} possesses half integer versions as well. For this we set $M^\mu_q (\la)=  \cI_{L_\mu}^{L_\la}(R_{L_\mu})$ for any $\mu\comp n$ with $Y_\mu\leq Y_\la$.

\begin{Theorem}\label{main3} Set again $\la=(n-1,1)\vdash n$. Then  
\begin{enumerate}
\item[1.)] The $\cH_\la$-module $\Res^\cH_{\cH_\la}\rV$ decomposes into $q$-permutation modules as follows:
$$
\Res^\cH_{\cH_\la}\rV \cong \bigoplus_{k=1}^{\min(n,r)}s(r+1,k)M_q^{(n-k,1^k)}(\la).
$$
\item[2.)] Similarly 
$$
\cR^G_{L_\la}(\fT^r_q) \cong \bigoplus_{k=1}^{\min(n,r)}s(r+1,k)\cI_{L_{(n-k,1^k)}}^{L_\la}(\leftindex_{L_{(n-k,1^k)}}{R}).
$$
\end{enumerate}
\end{Theorem}

\begin{proof} For the proof of 1.) note that theorem \ref{Hecke on tensor space} implies
\begin{equation*}
\begin{split}
\Res^\cH_{\cH_\la}\rV
& \cong  \bigoplus_{k=1}^{\min(n,r)}s(r,k)\Res^\cH_{\cH_\la}(M_q^{(n-k,1^k)}(\la))\\ 
&\cong  \bigoplus_{k=1}^{\min(n,r)}s(r,k)\Res^\cH_{\cH_\la}\Ind_{\cH_{(n-k,1^k)}}^\cH(Rx_{(n-k,1^k)}).
\end{split}
\end{equation*}
Applying  Mackey decomposition \ref{Mackey} yields the desired result. 

Alternatively, inspect the induction step in the proof of  \ref{indres} replacing $r$ by $r+1$ and noting, that 
\begin{equation*}
\begin{split}
V^{\otimes (r+1)}
&\cong \IR_{n-1}^n(\rV)\\ 
&= \Ind_{\cH_\la}^\cH(\Res^\cH_{\cH_\la}\rV)\\
& \cong \Ind_{\cH_\la}^\cH(\bigoplus_{k=1}^{\min(n,r)} s(r,k)\Res^\cH_{\cH_\la} M^{(n-k,1^k)}_q)\\
& =  \Ind_{\cH_\la}^\cH(\bigoplus_{k=1}^{\min(n,r)}s(r,k)\Res^\cH_{\cH_\la}\Ind_{(n-k,1^k)}^\cH(Rx_{(n-k,1^k)})).
\end{split}
\end{equation*}
Now equation  \ref{Mackey 4} in the proof of \ref{indres} yields the assertion. Part 2.) follows  easily by applying the Hecke functors \ref{functors} to the decomposition in the first part. 
\end{proof}

In [\cite{donkin 2}] Donkin proved the bicentralizer property for finite direct sums $M$ of $q$-permutation modules $M^\mu, \mu\vdash n,$ for Iwahori-Hecke algebras $\cH = \cH_{R,q}(\fS_n)$, where $R$ is a field or a Dedekind domain and the partitions $\mu$  labelling the summands satisfy a certain combinatorial condition, (namely that the "coarsening" of the partitions labelling the summands is cosaturated, for details see [\cite{donkin 2}]). Thus he showed in this case, that $\cH$ and $\cA = \End_\cH(M)$ are mutual centralizing algebras in $\End_R(M)$, that is $\End_\cA(M) = \cH$ as well. The set of hook partitions labelling the $q$-permutation modules appearing as direct summands in $\rV$ do not depend by \ref{Hecke on tensor space} and \ref{main3} on the choice of the unit $q$. Moreover  Donkin showed for $q=1$  in  [\cite{donkin 2}, 6.3], that their coarsening is cosaturated and hence $\rV$ satisfies Schur-Weyl duality. 

This implies now immediately:

\begin{Theorem}\label{main4} Let $R$ be a field or a Dedekind domain. Set $\cH = \cH_{R,q}(\fS_n)$ and $\cH_{R,q}(\fS_{n-1}) \cong \cH_\la$ for $\la=(n-q,1)\vdash n$. Then $\cH$ and $\part$ respectively $\cH_\la$ and $\parth$ acting on tensor space $\rV$ satisfy Schur-Weyl duality.

\hfill${\square}$\end{Theorem}

\section{Concluding remarks}

The decompositions  \ref{Hecke on tensor space} and \ref{main3} of  tensor space $\rV$ as $\cH$-module respectively as $\cH_\la$-module, where again $\la=(n-1,1)\vdash n,$ allows us to exhibit an $R$-basis of the $q$-partition algebra $\part$ and the half integer $q$-partition algebra $\parth$. This is derived from the standard basis of the {\bf $q$-Schur algebra} $\cS_q(n,R) = \End_\cH(\bigoplus_{\mu\comp n}M^\mu_q)$, which plays an important role in the representation theory of general linear groups, (see e.g. [\cite{mathas 1}]). The basic idea is the following:

We have as $R$-modules, using general principles on $\Hom$-spaces:
\begin{equation}\label{qSchur basis integer}
\begin{split}
\part
& = \Hom_\cH(\bigoplus_{k=1}^{\min(n,r)} s(r,k) M_q^{(n-k,1^k)},\bigoplus_{k=1}^{\min(n,r)} s(r,k) M_q^{(n-k,1^k)})\\
&\cong \bigoplus_{k,\ell=1}^{\min(n,r)}s(r,k)s(r,\ell)\Hom_\cH( M_q^{(n-k,1^k)}, M_q^{(n-\ell,1^\ell)}),
\end{split}
\end{equation}
where a basis  of $\Hom_\cH( M_q^{(n-k,1^k)}, M_q^{(n-\ell,1^\ell)})$ as $R$-module is given in equation \ref{qSchurbasis} in  the list of facts \ref{q-permmods}. A similar formula can be obtained for the half integer $q$-partition algebra $\parth$. In a forthcoming paper we shall investigate these bases and exhibit in particular its action on the $R$-basis $\cE^r$ of tensor space $\rV$.

As pointed out in the introduction the classical partition algebra $P_r(Q)$  is defined as associative $R$-algebra, $Q\in R$,  with a 
basis consisting of diagrams to set partition of $\{1,2,\ldots,2r\}$ and multiplication given by diagram concatenation. Xi proved in [\cite{Xi}], that $P_r(Q)$ is cellular, if $R$ is a field and $Q\in R$. It is a generalization both of the Brauer algebra, and also of the Temperley–Lieb algebra. Moreover Martin showed in [\cite{martin 3}], that $P_r(Q)$ is quasi-hereditary for fields $R$ of characteristic $0$ and $0\neq Q\in R$. This is not true in general for fields $R$ of positive characteristic, Xi produced in [\cite{Xi}] a counterexample.

If $Q=n\in\N$ is a natural number, then $\End_{R\fS_n}(\rV)$ is always epimorphic image of $P_r(n)$ and precisely an isomorphic image if $n\geq 2r$. In [\cite{donkin 1}] Donkin actually proved a general result, generalising in case $n\geq 2r$  the result of Xi above and determining precisely, when partition algebras over fields are quasi hereditary. He showed in particular that endomorphism rings of finite direct sums of permutation modules on Young subgroups of $\fS_n$ are cellular. It seems very likely that his methods can be generalised to Iwahori-Hecke algebras and their $q$-permutation modules, determining precisely, when $q$-partition algebras over fields are quasi hereditary. 

We conjecture that here is a tangle algebra $\widetilde{\cP}$ with a basis labelled by set partitions of $\{1,2,\ldots,2r\}$ which specializes to the partition algebra $P_r(Q)$ for certain values of the involved parameters and contains the Birman–Murakami–Wenzl algebra as subalgeba. The $q$-partition algebra $\part$ should be epimorphic image of specialised algebras $\widetilde{\cP}$ and be isomorphic image of those for $n\geq 2r$. Obviously it seems likely, that over fields, $\widetilde{\cP}$ should be cellular and over fields of characteristic $0$ even quasi-hereditary.

 \section*{Declarations}
\begin{itemize}
\item Ethical Approval: Not applicable
\item Competing Interest: Nil
\item Authors contribution: Both authors contributed equally
\item Funding: The first author's research was partially supported by DST-SERB under grants MATRIX MTR/2017/000424 and Power SPG/2021/004200. The first author would like to extend her gratitude to the Humboldt Foundation and Prof. Steffen Koenig for their support of her stay at University of Stuttgart.
\item Availability of data and material: Not applicable
\end{itemize}

\providecommand{\bysame}{\leavevmode ---\ }
\providecommand{\og}{``} \providecommand{\fg}{''}
\providecommand{\smfandname}{and}
\providecommand{\smfedsname}{\'eds.}
\providecommand{\smfedname}{\'ed.}
\providecommand{\smfmastersthesisname}{M\'emoire}
\providecommand{\smfphdthesisname}{Th\`ese}

\end{document}